\documentclass[12pt,reqno]{amsart}
\usepackage{latexsym,amsmath,amsfonts,amssymb}
\theoremstyle{plain}
\newtheorem{Thm}{Theorem}

\theoremstyle{definition}

\theoremstyle{remark}

\def\N{\mathbb N}

\def\F{\mathbb F}

\def\1{{\bf 1}}

\begin{document}
\title{An analogue of the prime number theorem for finite fields}
\author{Hao Pan}
\email{haopan1979@gmail.com}
\author{Zhi-Wei Sun}
\email{zwsun@nju.edu.cn}
\address{Department of Mathematics, Nanjing University,
Nanjing 210093, People's Republic of China}
\subjclass[2010]{Primary  11T06; Secondary 11B73}
\keywords{Irreducible polynomial, Prime number theorem, Stirling number of the second kind}
\begin{abstract}
We prove an analogue of the prime number theorem for finite fields.
\end{abstract}

\maketitle
Let $\pi(x)$ count all primes not exceeding $x$. The classical prime theorem asserts that
\begin{equation}\label{pnt}
\pi(x)=\frac{x}{\ln x}\cdot(1+o(1))
\end{equation}
as $x\to\infty$. In fact,  for every $m\geq 1$, we have the asymptotic expansion
\begin{equation}\label{pntg}
\pi(x)=\frac{x}{\ln x}+1!\cdot\frac{x}{(\ln x)^2}+2!\cdot\frac{x}{(\ln x)^3}+\cdots+m!\cdot\frac{x}{(\ln x)^m}+O_m\bigg(\frac{x}{(\ln x)^{m+1}}\bigg),
\end{equation}
where $O_m$ means that the implied constant only depends on $m$.

Suppose that $q$ is a power of a prime. Let $\F_q$ denote the finite field with $q$ elements.
Then the irreducible polynomials over $\F_q$ are correspond to the primes in $\N$. For $f(x)\in\F_q[x]$, define the norm of $f(x)$
$$
N(f)=q^{\deg f}.
$$
Let $\pi_q(x)$ count all monic irreducible polynomials in $\F_q[x]$ whose norm is not greater than $x$, i.e.,
$$
\pi_q(x)=\#\{f(x)\in\F_q[x]:\, f(x)\text{ is monic and irreducible, and }\deg f\leq\log_qx\}.
$$
It is well-know that the number of the monic irreducible polynomials in $\F[q]$ with  the degree $n$ is
$$
\frac1n\sum_{d\mid n}\mu\bigg(\frac nd\bigg)q^d,
$$
where $\mu$ is the M\"obius function. So substituting $x=q^n$, we have
$$
\pi_q(x)=\sum_{k=1}^n\frac1k\sum_{d\mid k}\mu\bigg(\frac kd\bigg)q^d.
$$

In \cite{KS90},  Kruse and Stichtenoth proved that
\begin{equation}\label{qpnt}
\pi_q(x)=\frac{q}{q-1}\cdot\frac{x}{\log_q x}\cdot(1+o(1))
\end{equation}
as $x=q^n\to\infty$. Note that
$$
\lim_{q\to 1}\frac{q}{q-1}\cdot\frac{x}{\log_q x}=\frac{x}{\ln x}\lim_{q\to 1}\frac{q}{q-1}\cdot\ln q=\frac{x}{\ln x}.
$$
So (\ref{qpnt}) can be viewed as a $q$-analogue of (\ref{pnt}). Subsequently, Wang and Kan \cite{WK10} obtained the following extension of  (\ref{qpnt}):
\begin{equation}\label{pnte}
\pi_q(x)=\frac{q}{q-1}\cdot\frac{x}{\log_q x}+\frac{q}{(q-1)^2}\cdot\frac{x}{(\log_q x)^2}+O\bigg(\frac{x}{(\log_q x)^3}\bigg),
\end{equation}
where $x=q^n\to\infty$.

On the other hand, supposing that $p$ is prime and  $x=p^{n-1}$, Pollack \cite{P10} gave an asymptotic expansion for $\pi_p(x)$:
\begin{equation}
\pi_p(x)=\frac{p}{p-1}\cdot\frac{x}{n}+\sum_{j=2}^m\frac{A_{p,j}x}{n^j}+O\bigg(\frac{\sqrt{x}}{n}+\frac{A_{p,m+2}x}{n^{m+1}}+\sum_{j=1}^m\frac{A_{p,j}}{n}\bigg),
\end{equation}
where the implied constant in $O$  is absolute and
$$
A_{p,j}=\sum_{k=1}^\infty\frac{k^{j-1}}{p^{k-1}}.
$$
Of course, as Pollack had pointed out, his discussions can be easily generalized for $\pi_q(x)$. However, since those $A_{p,j}$ are also infinite series, it is not convenient to compute.

So in this short note, we shall give another asymptotic expansion of $\pi_q(x)$, which can be viewed as a $q$-analogue of (\ref{pnte}).
Define the polynomial
$$
{\mathcal S}_n(x)=\sum_{k=0}^n(-1)^{j-k}k!S(n,k)x^k,
$$
where the Stirling number of the second kind $S(n,k)$ is given by
$$
t^n=\sum_{k=0}^nk!S(n,k)\binom{t}{k}.
$$
\begin{Thm}
Suppose that $q$ is a power of a prime and $x$ is a power of $q$. Then
\begin{equation}\label{mi}
\pi_q(x)=\frac{q}{q-1}\cdot\frac{x}{\log_q x}+\frac{1}{q-1}\sum_{j=1}^{m-1}{\mathcal S}_j\bigg(\frac{q}{q-1}\bigg)\cdot\frac{x}{(\log_q x)^{j+1}}+R_m,
\end{equation}
where
$$
|R_m|\leq\frac{11\sqrt{x}}{\log_q x}+m!\bigg(\frac{3^{m+3}x}{q(\log_q x)^{m+1}}+\frac{2^{m+1}}{\log_q x}\bigg).
$$
\end{Thm}
\begin{proof}
For convenience, we write $f=\Theta_C(g)$ provided that $|f(n)|\leq C|g(n)|$ for every $n\geq 1$. It is not difficult to show that
$$
q^{k}-2q^{k/2}\leq\sum_{d\mid k}\mu\bigg(\frac kd\bigg)q^d\leq q^k.
$$
And since $q\geq 2$, using induction on $n$ and a simple calculation, we can easily deduce that
$$
\sum_{k=1}^n\frac{q^{k/2}}{k}\leq\frac{5.5q^{n/2}}{n}
$$
for every $n\geq 1$.
Hence we have
$$
\pi_q(x)=\sum_{k=1}^n\frac1k\sum_{d\mid k}\mu\bigg(\frac kd\bigg)q^d=
\sum_{k=1}^n\frac{q^k}k+
\Theta_{2}\bigg(\sum_{k=1}^n\frac{q^{k/2}}k\bigg)=
\sum_{k=1}^n\frac{q^k}k+
\Theta_{11}\bigg(\frac{q^{n/2}}{n}\bigg).$$

Evidently,
$$
\sum_{k=1}^n\frac{q^k}{k}=
\sum_{k=0}^{n-1}\frac{q^{n-k}}{n-k}
=\frac{q^n}{n}\sum_{k=0}^{n-1}\frac{q^{-k}}{1-k/n}\\
=\frac{q^n}{n}\sum_{k=0}^{n-1}\frac1{q^{k}}\sum_{j=0}^\infty\frac{k^j}{n^j}.
$$
We have
$$
\sum_{k=1}^{n-1}\frac1{q^{k}}\sum_{j=m}^\infty\frac{k^j}{n^j}=\frac{1}{n^m}\sum_{k=1}^{n-1}\frac{k^m}{q^k}\cdot\frac{1}{1-k/n}.
$$
Note that
$$
\sum_{k=1}^\infty\frac{k^m}{q^k}\leq \frac{1}{q}\sum_{k=0}^\infty\frac{(k+1)(k+2)\cdots(k+m)}{q^k}=\frac{m!}{q(1-1/q)^{m+1}}.
$$
Hence
$$
\sum_{1\leq k<n/2}\frac{k^m}{q^k}\cdot\frac{1}{1-k/n}\leq 2\sum_{k=1}^\infty\frac{k^m}{q^k}\leq\frac{2^{m+2}m!}{q}.
$$
Furthermore, we have
\begin{align*}
\sum_{n/2\leq k\leq n-1}\frac{k^m}{q^k}\cdot\frac{1}{1-k/n}\leq&n\sum_{n/2\leq k\leq n-1}\frac{k^m}{q^k}\leq \frac{n}{1.3^{n/2}}\sum_{n/2\leq k\leq n-1}\frac{k^m}{(0.75q)^k}\\
\leq&\frac{n}{1.3^{n/2}}\sum_{k=1}^{\infty}\frac{k^m}{(0.75q)^k}\leq\frac{n}{1.3^{n/2}}\cdot\frac{m!}{0.75q(1-1/(0.75q))^{m+1}}.
\end{align*}
It is easy to check that
$$
\bigg(\frac{0.75q}{0.75q-1}\bigg)^{m+1}\leq3^{m+1}\quad
\text{and}\quad
\frac{n}{1.3^{n/2}}\leq 3
$$
for every $q\geq 2$ and $n,m\geq 1$. So
$$
\sum_{k=1}^{n-1}\frac{k^m}{q^k}\cdot\frac{1}{1-k/n}\leq \frac{2^{m+2}m!}{q}+\frac{2\cdot 3^{m+2}m!}{q}\leq \frac{3^{m+3}m!}{q}.
$$
Thus
\begin{align*}
\sum_{k=1}^n\frac{q^k}{k}=&\frac{q^n}{n}\sum_{k=0}^{n-1}\frac1{q^{k}}\sum_{j=0}^{m-1}\frac{k^j}{n^j}+
\frac{q^n}{n}\sum_{k=1}^{n-1}\frac1{q^{k}}\sum_{j=m}^\infty\frac{k^j}{n^j}\\
=&\sum_{j=0}^{m-1}\frac{q^n}{n^{j+1}}\sum_{k=0}^{n-1}\frac{k^j}{q^{k}}+
\Theta_{1}\bigg(\frac{3^{m+3}m!q^{n-1}}{n^{m+1}}\bigg).
\end{align*}

Furthermore, we have
$$
\sum_{k=1}^{n-1}\frac{k^j}{q^{k}}=\sum_{k=1}^{\infty}\frac{k^j}{q^{k}}-\sum_{k=n}^{\infty}\frac{k^j}{q^{k}}=\sum_{k=1}^{\infty}\frac{k^j}{q^{k}}+\Theta_1\bigg(\frac{n^{j}}{q^n}\sum_{k=1}^{\infty}\frac{k^j}{q^{k-1}}\bigg)=\sum_{k=1}^{\infty}\frac{k^j}{q^{k}}+\Theta_{2}\bigg(\frac{n^{j}}{q^n}\cdot 2^jj!\bigg).
$$
So
\begin{align*}
&\frac{q^n}{n^{j+1}}\sum_{k=1}^{n-1}\frac{k^j}{q^{k}}=\frac{q^n}{n^{j+1}}\sum_{k=1}^{\infty}\frac{(-1)^j}{q^{k}}\cdot(-k)^j+\Theta_{2}\bigg(\frac{2^jj!}{n}\bigg)\\
=&\frac{q^n}{n^{j+1}}\sum_{k=1}^{\infty}\frac{(-1)^j}{q^{k}}\sum_{i=0}^ji!S(j,i)\binom{-k}{i}+\Theta_{2}\bigg(\frac{2^jj!}{n}\bigg)\\
=&\frac{q^{n-1}}{n^{j+1}}\sum_{i=0}^j(-1)^{j-i}i!S(j,i)\sum_{k=1}^{\infty}\frac{1}{q^{k-1}}\binom{i+k-1}{k-1}+\Theta_{2}\bigg(\frac{2^jj!}{n}\bigg)\\
=&
\frac{q^{n-1}}{n^{j+1}}\sum_{i=0}^j(-1)^{j-i}i!S(j,i)\cdot\frac{q^{i+1}}{(q-1)^{i+1}}+\Theta_{2}\bigg(\frac{2^jj!}{n}\bigg).
\end{align*}
Thus we get that
\begin{align*}
&\sum_{k=1}^n\frac{q^k}{k}=\frac{q^{n+1}}{n(q-1)}+\sum_{j=1}^{m-1}\frac{q^n}{n^{j+1}}\sum_{k=0}^{n-1}\frac{k^j}{q^{k}}+\Theta_{1}\bigg(\frac{3^{m+3}m!q^{n-1}}{n^{m+1}}\bigg)\\
=&\frac{q^{n+1}}{n(q-1)}+\sum_{j=1}^{m-1}\frac{q^{n}}{n^{j+1}}\sum_{i=0}^j\frac{(-1)^{j-i}i!S(j,i)q^{i}}{(q-1)^{i+1}}+\Theta_{1}\bigg(\frac{3^{m+3}m!q^{n-1}}{n^{m+1}}\bigg)+\Theta_{2}\bigg(\frac{2^mm!}{n}\bigg).
\end{align*}
This concludes our proof. 
\end{proof}

\end{document}